\documentclass{conm-p-l}

\usepackage{amssymb}

\usepackage{graphicx}

\usepackage{}

\newtheorem{theorem}{Theorem}[section]
\newtheorem{lemma}[theorem]{Lemma}
\newtheorem{corol}[theorem]{Corollary}
\newtheorem{prop}[theorem]{Proposition}
\newtheorem{conj}[theorem]{Conjecture}
\newtheorem{problem}[theorem]{Problem}
\newtheorem{ques}[theorem]{Problem}

\theoremstyle{definition}
\newtheorem{definition}[theorem]{Definition}
\newtheorem{example}[theorem]{Example}

\theoremstyle{remark}
\newtheorem{rem}[theorem]{Remark}

\newcommand{\brho}{{\boldsymbol{\rho}}}
\newcommand{\btau}{{\boldsymbol{\tau}}}

\newcommand{\bR}{{\mathbb R}}
\newcommand{\bZ}{{\mathbb Z}}

\newcommand{\bQ}{{\mathbb Q}}
\newcommand{\bN}{{\mathbb N}}

\newcommand{\bfi}{{\mathbf i}}
\newcommand{\bfj}{{\mathbf j}}
\newcommand{\bfk}{{\mathbf k}}
\newcommand{\bfw}{{\mathbf w}}
\newcommand{\bfz}{{\mathbf z}}

\newcommand{\SD}{{\mathcal D}}
\newcommand{\SE}{{\mathcal E}}

\newcommand{\SH}{{\mathcal H}}

\newcommand{\M}{{\mathcal M}}
\newcommand{\SR}{{\mathcal R}}

\newcommand{\ST}{{\mathcal T}}
\newcommand{\SW}{{\mathcal W}}

\def\ga{\alpha}
\def\gb{\beta}
\def\gr{\rho}

\def\gS{\Sigma}
\def\gY{\Psi}

\def\i{\mathbf{i}}
\def\j{\mathbf{j}}
\def\k{\mathbf{k}}

\def\sgp{\mathrm{sgp}}
\def\bleq{\sim}
\newcommand{\card}{{\rm card\,}}
\newcommand{\diam}{{\rm diam\,}}
\newcommand{\rank}{{\rm rank}}

\newcommand{\lra}{{\longrightarrow}}

\newcommand{\mhsp}{\hspace{2em}}

\numberwithin{equation}{section}

\begin{document}

% \title[short text for running head]{full title}
\title[Lipschitz equivalence of Cantor sets]{Lipschitz Equivalence of Self-Similar
 Sets: Algebraic and Geometric Properties}

%    Only \author and \address are required; other information is
%    optional.  Remove any unused author tags.

%    author one information
% \author[short version for running head]{name for top of paper}
\author{Hui Rao}
\address{Department of Mathematics, Hua Zhong Normal University, Wuhan 430079, China}
\email{hrao@mail.ccnu.edu.cn}
\thanks{The research of Rao is supported by the
NSFC grant 11171128.}

%    Information for second author
\author{Huo-Jun Ruan}
\address{Department of Mathematics, Zhejiang University, Hangzhou 310027,
China} 
\email{ruanhj@zju.edu.cn}
\thanks{The research of Ruan was supported in part by NSFC grant 11271327,
ZJNSFC grant Y6110128 and the Fundamental Research Funds for the Central Universities of China.}

\author{Yang Wang}
\address{Department of Mathematics, Michigan State
University, East Lansing, MI 48824, USA} 
\email{ywang@math.msu.edu}
\thanks{The research of Wang was supported in part by
NSF Grant DMS-1043034 and DMS-0936830}

\thanks{Corresponding author: Huo-Jun Ruan}

%\author{}
%\address{}
%\curraddr{}
%\email{}
%\thanks{}

%    author two information
%\author{}
%\address{}
%\curraddr{}
%\email{}
%\thanks{}

%\subjclass[2000]{Primary }
%    The 2010 edition of the Mathematics Subject Classification is
%    now available.  If you are citing a classification from the
%    new scheme, use the following input coding instead.

\date{}


\begin{abstract}
 In this paper we provide an up-to-date survey on the study of Lipschitz equivalence of
 self-similar sets. Lipschitz equivalence is an important property in fractal geometry
 because it preserves many key properties of fractal sets. A fundamental result by
 Falconer and Marsh  [On the Lipschitz equivalence of Cantor sets, \textit{Mathematika}, \textbf{39} (1992), 223--233]
 establishes conditions for Lipschitz equivalence based on the algebraic properties of the contraction ratios of the self-similar sets. Recently there has been other substantial progress in the field. This paper is a comprehensive survey of the field.
 It provides a summary of the important and interesting results in the field. In addition we provide detailed discussions on several important techniques that have been used to
 prove some of the key results. It is
our hope that the paper will provide a good overview of major results and techniques, and a friendly entry point for anyone who is interested in
studying problems in this field.
\end{abstract}

\maketitle

\section{Introduction}
\label{intro}

In the study of fractal geometry a fundament problem is to find ways that measure the similarity or difference of fractal sets. The concept of dimension, whether it is the
Hausdorff dimension or the box counting dimension, is widely used for such a purpose:
Two sets having different dimensions are considered to be unalike. However for
measuring differences dimension by itself is quite inadequate. Two compact sets, even with the
same dimension, may in fact be quite different in many ways. Thus it is natural to
seek a suitable quality that would allow us to tell whether two
fractal sets are ``similar''. Generally, Lipschitz equivalence is thought to be such a quality.
In \cite{FaMa92} it was pointed out that while topology may be regarded as the study
of equivalence classes of sets under homeomorphism, fractal
geometry is sometimes thought of as the study of equivalence classes
under bi-Lipschitz mappings. More restrictive maps such as isometry
tend to lead to rather uninteresting equivalent classes, while far less restrictive
maps such as general continuous maps take us completely out of geometry
into the realm of pure topology (see~\cite{Gromo07}). Bi-Lipschitz maps offer
a good balance, which lead to equivalent classes that are interesting and intriguing
both geometrically and algebraically.

There has been notable progress on the study of bi-Lipschitz equivalence classes,
especially in recent years. Yet much is still unknown, and this progress has
led to more unanswered questions.
The goal of this paper is to provide {\em a comprehensive survey of the area}. It is
our hope that the paper will provide a good overview of major results and techniques, and a friendly entry point for anyone who is interested in
studying problems in this field.

Let $E, F$ be compact sets in $\bR^d$. We say that $E$ and $F$ are
{\em Lipschitz equivalent}, denoted by $E \bleq F$, if there
exists a bijection $f: E \lra F$ which is \textit{bi-Lipschitz},
 i.e. there exists a constant $C>0$ such that
$$
     C^{-1}|x-y|\leq |f(x)-f(y)| \leq C|x-y|
$$
for all $x,y\in E$. The general problem we consider is to find conditions
under which the two sets $E$ and $F$ are Lipschitz equivalent.

Recall that in general we characterize
a self-similar set as the attractor of an {\em iterated function
system (IFS)}. Let $\{\phi_j\}_{j=1}^m$ be an IFS on $\bR^d$ where
each $\phi_j$ is a contractive similarity with contraction ratio
$0<\rho_j<1$. The attractor of the IFS is the unique nonempty
compact set $F$ satisfying $F = \bigcup_{j=1}^m \phi_j(F)$, see
\cite{Hut81}. We say that the attractor $F$ is {\em dust-like}, or
alternatively, the IFS $\{\phi_j\}$ satisfies the {\em strong
separation condition (SSC)}, if the sets $\{\phi_j(F)\}$ are
disjoint. We remark that by definition, ``dust-like self-similar set''
is not the same as a ``totally disconnected self-similar set''.
 It is well known that if $F$ is dust-like then the
Hausdorff dimension $s=\dim_H(F)$ of $F$ satisfies $\sum_{j=1}^m
\rho_j^s =1$.

Now for any $\rho_1, \dots, \rho_m\in(0,1)$ with $\sum_{j=1}^m
\rho_j^d<1$, we will call $\brho=(\rho_1, \dots, \rho_m)$ a
\emph{contraction vector}, and use the notation
$\SD(\brho)=\SD(\rho_1, \dots, \rho_m)$ to denote the set of all
dust-like self-similar sets that are the attractor of some IFS with
contraction ratios $\rho_j, j=1,\dots, m$ on $\bR^d$.   Clearly all sets in
$\SD(\brho)$ have the same Hausdorff dimension, which we denote by
$s=\dim_H\SD(\brho)$. The following property is well known, see e.g.
\cite{RRX06}.

\begin{prop}\label{prop-1.1}
  Any two sets in $\SD(\brho)$ are Lipschitz equivalent.
\end{prop}
This result tells us that in the dust-like setting all that matters is the
contraction vector. The translations in the similitudes in the IFS do not
matter. In fact, all sets in $\SD(\brho)$ are Lipschitz equivalent to
a symbolic space defined by $\brho$.
For any $m \geq 1$ let $\Sigma_m$ denote the set of
infinite words $\bfw =i_1i_2i_3\cdots$ where each
$i_j\in \{1,2, \dots, m\}$. For such a $\bfw\in\Sigma_m$ we use the
notation $\bfw(k)=i_k$ and $[\bfw]_k =i_1i_2\cdots i_k$. For any
$\brho=(\rho_1, \rho_2, \dots, \rho_m)$, $0<\rho_j<1$, we can define
a metric ${\mathbf d}_\brho(.,.)$ on $\Sigma_m$ as follows: Let
$\bfz,\bfw\in\Sigma_m$. If $\bfz(1)\neq \bfw(1)$ then set
${\mathbf d}_\brho (\bfz,\bfw)=1$; otherwise set ${\mathbf d}_\brho
(\bfz,\bfw)=\brho_{[\bfz]_k}$, where $[\bfz]_k=[\bfw]_k$ but
$\bfz(k+1) \neq \bfw(k+1)$, and $\brho_{[\bfz]_k}:=\prod_{j=1}^k
\rho_{\bfz(j)}$. It is well known that $(\Sigma_m, {\mathbf d}_\brho)$ is a metric space.
The following is easy to prove:

\begin{prop}  \label{prop-1.2}
     Let $\brho=(\rho_1, \dots, \rho_m)$ be a contraction vector
     and $E \in \SD(\brho)$.
     Then there exists a bi-Lipschitz map from
     $(\Sigma_m, {\mathbf d}_\brho)$ to $E$.
\end{prop}

%\vspace{2mm}
\begin{rem} 
It was noted  in \cite{RRW12} that the proof for
Proposition \ref{prop-1.2} leads to the following simple but interesting result:
     Assume that $\SD(\rho_1, \dots, \rho_m)$ and $\SD(\tau_1, \dots, \tau_n)$ are
     Lipschitz equivalent. Let $s =\dim_H \SD(\rho_1, \dots, \rho_m)$.
     Then for any $r>s$, $\SD(\rho_1^r, \dots, \rho_m^r)$ and $\SD(\tau_1^r, \dots, \tau_n^r)$ are
     also Lipschitz equivalent.
\end{rem}
%\vspace{2mm}

Proposition \ref{prop-1.1} gives a ``trivial condition'' for Lipschitz equivalence.
A generalization of this ``trivial condition'' is when one contraction ratio is
derived from another.

Let $\Sigma_m^*:=
\bigcup_{k=1}^\infty \{1,2,\dots,m\}^k$. For any word
$\bfi=i_1\cdots i_k\in\Sigma_m^*$,  we denote $[\bfi]=\{\bfi\bfw:~\bfw\in \Sigma_m\}$ and call it a \emph{symbol cylinder}.
A finite set $\{\bfj_1,\dots, \bfj_n \} \subset \Sigma_m^*$ is called a \emph{cut set} of $\Sigma_m$
if the symbol cylinders $[\bfj_1], \dots, [\bfj_n]$ tile $\Sigma_m$, i.e., they are disjoint and their union is $\Sigma_m$.

Let $\brho=(\rho_1,\dots, \rho_m)$ and $\btau=(\btau_1,\dots, \btau_n)$ be two contraction vectors. We say that  $\btau$
is {\em derived from} $\brho$ if there exists a cut set $\{\bfj_1,\dots, \bfj_n \}$ of $\Sigma_m$ such that
$\btau=(\brho_{\bfj_1},\dots, \brho_{\bfj_n})$, where $\brho_{i_1\cdots i_k}=\rho_{i_1}\cdots \rho_{i_k}$.

\begin{definition}
     Let $\brho$ and $\btau$ be two contraction vectors.
     % We
%     say $\btau$ is {\em dervied} from $\brho$ if there is an IFS $\Phi =\{\phi_j\}_{j=1}^m$
%     with contraction vector $\brho$ satisfying the SSC and another IFS $\Psi =\{\psi_i\}_{i=1}^n$
%     with contraction vector $\btau$ such that $\Psi$ is derived from $\Phi$.
%
We say $\brho$ and
     $\btau$ are {\em equivalent}, denoted by $\brho \sim \btau$, if
     there exists a sequence
$$
   \brho=\brho_1, \brho_2, \dots, \brho_N=\btau
$$
such that $\brho_{j+1}$ is derived from $\brho_j$ or vice versa for
$1 \leq j <N$.
\end{definition}
%
%
% The equivalence of two contraction vectors constitutes another ``trivial condition'' for Lipschitz
% equivalence:

\begin{prop}  \label{prop-1.3}
    Assume that $\brho$ is equivalent to $\btau$. Then $\SD(\brho) \bleq \SD(\btau)$.
\end{prop}

\begin{proof} We need only show the conclusion holds when $\btau$ is derived from $\brho$. Suppose $E\in \SD(\brho)$ is the attractor of the IFS $\Phi =\{\phi_j\}_{j=1}^m$, then $E$ is also the attractor of the IFS $\{\phi_{\bfj_1},\dots, \phi_{\bfj_n}\}$, where $\phi_{i_1\cdots i_k}:=\phi_{i_1}\circ\cdots\circ\phi_{i_k}$.
Hence $\SD(\brho)$ and $\SD(\btau)$ have a common element $E$, and they are equivalent.
\end{proof}

The central question in the study of
Lipschitz equivalence of dust-like Cantor sets is: Under what conditions are two dust-like Cantor
sets Lipschitz equivalent even if they have different contraction vectors? Are there any ``nontrivial conditions''
that also lead to equivalence?

\begin{problem}\label{ques:central}
Find nontrivial sufficient conditions and necessary
conditions on $\brho$ and $\btau$ such that $\SD(\brho)\sim
\SD(\btau)$. In particular, is it possible that $\SD(\brho)
\bleq \SD(\btau)$ but $\brho$ and $\btau$ are not equivalent?
\end{problem}

As it turns out, among the known results concerning this central question,
the algebraic properties of contraction vectors have played a
fundamental role. This is a main focus of this survey.

One of the very first and most fundamental results in this area
is the following theorem, proved by
Falconer and Marsh (\cite{FaMa92}, Theorem 3.3),  that establishes a connection to
the algebraic properties of the contraction ratios:

\begin{theorem}[Falconer and Marsh]\label{theo:FaMa}
  Let $\SD(\brho)$ and $\SD(\btau)$ be Lipschitz equivalent, where $\brho=(\rho_1, \dots, \rho_m)$
  and $\btau=(\tau_1, \dots, \tau_n)$ are two contraction vectors.
  Let $s=\dim_H \SD(\brho)=\dim_H \SD(\btau)$. Then
  \begin{itemize}
  \item[\rm (1)] $\bQ(\rho_1^s,\ldots,\rho_m^s)= \bQ(\tau_1^s,\ldots,\tau_n^s)$, where $\bQ(a_1,\ldots,a_m)$
  denotes the subfield of $\bR$ generated by $\bQ$ and $a_1,\dots, a_m$.
  \item[\rm (2)] There exist positive integers $p,q$ such that
  \begin{eqnarray*}
    &\sgp(\rho_1^p,\ldots,\rho_m^p)\subseteq \sgp(\tau_1,\ldots,\tau_n),     \\
    &\sgp(\tau_1^q,\ldots,\tau_n^q)\subseteq
    \sgp(\rho_1,\ldots,\rho_m),
  \end{eqnarray*}
  where $\sgp(a_1,\ldots,a_m)$ denotes the subsemigroup of
  $(\bR^+,\times)$ generated by $a_1,\ldots,a_m$.
  \end{itemize}
\end{theorem}
Using this theorem, it is trivial to construct
dust-like self-similar sets $E$ and $F$ such that $\dim_H E=\dim_H
F$ but $E$ and $F$ are not Lipschitz equivalent. For example,
let $E$ be the middle-third Cantor set and $F$ be the dust-like Cantor set
given by the IFS $\Phi:=\{\rho x, \rho x + \frac{1}{2}(1+\rho), \rho x+ 1-\rho\}$
where $\rho = 3^{-\log_2 3}$. Then $E$ and $F$ have the same dimension. However,
they are not Lipschitz equivalent by Theorem \ref{theo:FaMa}.

Along the direction of the theorem of Falconer and Marsh, several other
theorems have been established in recent years. These theorems further
establish connections between Lipschitz equivalence and algebraic properties of
the contractions. We shall discuss them, along
with several key techniques, later in this paper.

Another interesting question on Lipschitz equivalence, in a different direction,
considers the geometric structures of self-similar sets. Perhaps
the best known problem is the one proposed by David and Semmes (\cite{DS}, Problem~11.16):
\begin{problem}\label{prob1}
  Let $\phi_i(x):=x/5+(i-1)/5$ where $i\in \{1,\cdots,5\}$. Let $M$ and
  $M^\prime$ be the attractor of the IFS $\{\phi_1, \phi_3,\phi_5\}$ and
  the IFS $\{\phi_1,\phi_4,\phi_5\}$, respectively. Are $M$ and $M^\prime$ Lipschitz
  equivalent?
\end{problem}

\begin{figure}[htbp]
\begin{center}
  \scalebox{0.5}{\includegraphics{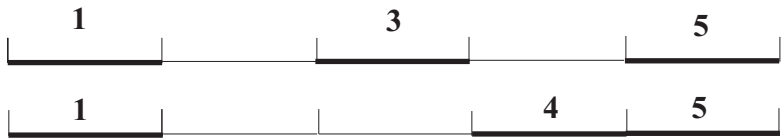}}
\end{center}
\caption{Basic intervals of the self-similar sets $M$ and $M'$}
\label{figure: DS-set}
\end{figure}

We call $M$ the $\{1,3,5\}$-set and $M^\prime$ the $\{1,4,5\}$-set. The problem is
generally known as the $\{1,3,5\}$-$\{1,4,5\}$ {\em problem}.
In this setting, $M$ is dust-like and $M^\prime$ has certain touching
structure, see Figure~\ref{figure: DS-set}. In this problem, the contraction
ratios are all identical so the difference lies entirely in the geometry of
the two IFSs. David and Semmes
conjectured that $M\not\sim M^\prime$. However, by examining graph-directed
structures of the attractors and introducing techniques to study
Lipschitz equivalence on these structures, Rao, Ruan and Xi
\cite{RRX06} proved that in fact $M\sim M^\prime$. Naturally one may
ask whether this result extends to the general setting, where we
consider the equivalence of two IFSs with the same contraction vector, but one
is dust-like while the other has some touching structure. We shall discuss this
problem in more details also later in the paper.

In other direction, some recent works are done on the Lipschitz equivalence
of $\lambda$-Cantor sets, which are self-similar sets with overlap.
We refer the readers to \cite{DH,GLWX}.

\section{Techniques for Lipschitz Equivalence of Dust-Like Cantor Sets}
\setcounter{equation}{0}

%Since the above work by Falconer and Marsh, there have been little
%progress in this direction as we know of.

\subsection{Techniques in \cite{FaMa92}}

In \cite{FaMa92} Falconer and Marsh had developed several important techniques
to study the Lipschitz equivalence of dust-like self-similar sets.
These techniques are now viewed as being fundamental to the area. These
techniques had allowed Falconer and Marsh to prove Theorem \ref{theo:FaMa}.

Let us first introduce some notation. Let $E$ be the attractor of the
IFS $\Phi=\{\phi_1, \dots, \phi_m\}$. For any word
$\bfi=i_1\cdots i_k\in\Sigma_m^*$, we call $k$ the length of the
word $\bfi$ and denote it by $|\bfi|$. Furthermore,  a {\em
cylinder} $E_\bfi$ is defined to be $E_\bfi =
\phi_\bfi(E):=\phi_{i_1}\circ\cdots\circ\phi_{i_k}(E)$.

In this section we consider the Lipschitz equivalence of two
dust-like self-similar sets $E$ and $F$ with the following setup: We
assume that $E\in \SD(\rho_1,\ldots,\rho_m)$  is the attractor of $\Phi=\{\phi_1, \dots,
\phi_m\}$  and
$F\in \SD(\tau_1,\ldots,\tau_n)$ is the attractor of $\Psi=\{\psi_1, \dots, \psi_n\}$. We also assume in
subsections~2.1 and 2.2 that $s=\dim_H E=\dim_H F$ and $f: E \lra F$
is a bi-Lipschitz map.

An important result is the following lemma:

\begin{lemma}[\cite{FaMa92}] \label{lem:FaMa-fundamental}
There exists an integer $n_0$ such that for any $\bfi \in
\Sigma_m^*$, there exist $\bfk, \bfj_1, \dots, \bfj_{p}
\in\Sigma_n^*$ such that $F_{\bfk\bfj_1}, \dots, F_{\bfk\bfj_{p}}$
are disjoint and
\begin{equation} \label{2.1}
   f(E_\bfi)=\bigcup_{r=1}^{p} F_{\bfk\bfj_r}\subset F_\bfk,
\end{equation}
where each $|\bfj_r|\leq n_0$. In particular
$\SH^s(f(E_\bfi))=\SH^s(F_\bfk)\sum_{r=1}^{p} (\btau_{\bfj_r})^s.$
\end{lemma}

\begin{rem}\label{rem:for-FM-fund}
\textnormal{The above lemma implies that a bi-Lipschitz map
must behave ``nicely'' by mapping a cylinder onto a union of cylinders. We can require $F_{\k}$ to be the smallest cylinder containing $f(E_{\i})$.
It is clear that we can also require each $|\bfj_r|=n_0$ in
the above lemma. Consequently the set
$\{\k,\bfj_1, \dots, \bfj_{p}\}$  is uniquely determined by
$\bfi$. We will write $p_\bfi$ for $p$ if necessary. We call this
unique decomposition to be the \emph{maximum decomposition} of
$f(E_\bfi)$ with respect to $F$ and $n_0$. From now on, we fix $n_0$
in this section. We remark that $p$ in (\ref{2.1}) is bounded since
$p\leq n^{n_0}$.}
\end{rem}

One of the key techniques in \cite{FaMa92} is the introduction
of a sequence of functions $g_k:\
E\lra \mathbb R$, given by
\begin{equation}  \label{eq:gk-def}
    g_k(x)=\frac{\SH^s(f(E_\bfi))}{\SH^s(E_\bfi)}
\end{equation}
for  $x\in E_\bfi$, where $\bfi\in\{1,\dots, m\}^k$.  This sequence plays
a crucial role in studying the Lipschitz equivalence of dust-like
Cantor sets. We shall abuse
notation by writing
$    g_k(E_\bfi) =\frac{\SH^s(f(E_\bfi))}{\SH^s(E_\bfi)}.$
It is easy to show that
\begin{equation}\label{eq:mart-prop}
    g_k(E_\bfi)=\sum_{i=1}^m \frac{\SH^s(E_{\bfi i})}{\SH^s(E_\bfi)} g_{k+1}(E_{\bfi i}).
\end{equation}
Furthermore, it is not difficult to prove:
\begin{lemma}[\cite{FaMa92}] \label{lem:FM-gRatio-finite}
    The set $\{\frac{g_{k+1}(x)}{g_k(x)}:~ x\in E, k\geq 1\}$ is finite.
\end{lemma}
%
%\subsection{Measure-preserving property}

An important observation is that $\{g_k\}$ form a martingale with respect to the
normalized Hausdorff measure $\SH^s$
and a suitable filtration.
By the Martingale Convergence Theorem the sequence of functions
$\{g_k\}$ converges almost everywhere with respect to $\SH^s$. However, note that
$g_k(x)$ take on only finitely many values by Lemma \ref{lem:FM-gRatio-finite}.
It follows that for almost every $x\in E$, there exists a $k_0$ such that for
$k\geq k_0$ we must have $g_k(x)=g_{k_0}(x)$. Using this result, Theorem~\ref{theo:FaMa} can be proved.

\subsection{Measure-preserving property}
In \cite{CP}, Cooper and Pignataro studied the order-preserving bi-Lipschitz functions between two dust-like Cantor subsets of $\bR$. They proved that such functions have certain measure preserving property. Xi and Ruan \cite{XiRu08} observed that this property also holds in more general case.

\begin{lemma}[\cite{CP, XiRu08}] \label{lem:XiRuan}
    There is a cylinder $E_{\bfi_0}$ and a constant $c>0$ such that $g_k(x)=c$ for all
    $x\in E_{\bfi_0}$ and $k\geq |\bfi_0|$.
\end{lemma}
\begin{proof} Set $T=\sup_{k\geq 1} \max_{|\bfi|=k}   g_k(E_{\bfi})$.
Since $f$ is bi-Lipschitz, we have $T<+\infty$.

If $\frac{g_{k+1}(x)}{g_k(x)}=1$ for all $x\in E$ and all $k\geq 1$,
then the lemma clearly holds. Otherwise set $\delta=\min
\left(\left\{|\frac{g_{k+1}(x)}{g_k(x)}-1|: ~x\in E, k\geq
1\right\}\setminus\{0\}\right)$. Then  $\delta>0$ by
Lemma~\ref{lem:FM-gRatio-finite}. Choose ${\bfi_0}$  such that
(with $\ell=|{\bfi_0}|$)
\begin{equation}   \label{eq:mart-Ei0}
     g_{\ell}(E_{{\bfi_0}})>T/(1+\delta).
\end{equation}
Then $\frac{g_{\ell+1}(E_{{\bfi_0}
j})}{g_\ell(E_{{\bfi_0}})}<1+\delta$ for all $j$ and hence
$\frac{g_{\ell+1}(E_{{\bfi_0} j})}{g_\ell(E_{{\bfi_0}})}\leq 1$ by
the definition of $\delta$.

Now  formula~(\ref{eq:mart-prop}) implies that
$\frac{g_{\ell+1}(E_{{\bfi_0} j})}{g_\ell(E_{{\bfi_0}})}= 1$ for all
$j$. Hence each $E_{{\bfi_0} j}$  satisfies (\ref{eq:mart-Ei0}) and
we can repeat the same argument with $E_{{\bfi_0} j}$ in place of
$E_{\bfi_0}$. Set $c=g_\ell(E_{{\bfi_0}})$ and the lemma is proved.
\end{proof}

This lemma means that the restriction of $f$ on $E_{\bfi_0}$ is
measure-preserving up to a constant. More precisely for any Borel
set $A\subset E_{\bfi_0}$ we have

\begin{equation}  \label{3.5}
    \frac{\SH^s(f(A))}{\SH^s(A)}=c=\frac{\SH^s(f(E_{\bfi_0}))}{\SH^s(E_{\bfi_0})}.
\end{equation}

We shall call any such cylinder $E_{\bfi_0}$ in Lemma~\ref{lem:XiRuan} a
\emph{stable cylinder} with respect to the map $f$. In the rest of this section we
fix a stable cylinder $E_{\bfi_0}$. Going back to
Lemma~\ref{lem:FaMa-fundamental} and Remark~\ref{rem:for-FM-fund},
for any $\bfi\in \Sigma_m^*$, there is a (unique) maximum
decomposition of $f(E_{\bfi_0\bfi})$ with respect to $F$ and $n_0$:
$$
   f(E_{\bfi_0\bfi})=\bigcup_{r=1}^{p_{\bfi_0\bfi}} F_{\bfk\bfj_r},
$$
where  $|\bfj_r|= n_0$. This allows us to prove the following observation, which serves
as a key result in the development of the {\em matchable condition} technique in \cite{RRW12}.

\begin{lemma}[\cite{RRW12}]  \label{lem:HM-Ratio-Finite}
  The set  $
      \M=\bigcup_{\bfi\in \Sigma_m^*}  \Bigl\{\frac{\SH^s(E_{\bfi_0\bfi})}{\SH^s(F_{\bfk\bfj_r})}:
        1\leq r\leq p_{\bfi_0\bfi}~\Bigr\}
  $  is finite. Consequently, the sets
  $$   \M'=\bigcup_{\bfi\in \Sigma_m^*}  \Bigl\{\frac{\diam E_{\bfi_0\bfi}}{\diam F_{\bfk\bfj_r}}:
        1\leq r\leq p_{\bfi_0\bfi}~\Bigr\} \;\; \mbox{and}\;\;
        \M''=\bigcup_{\bfi\in \Sigma_m^*}  \Bigl\{\frac{\brho_{\bfi_0\bfi}}{\btau_{\bfk\bfj_r}}:
        1\leq r\leq p_{\bfi_0\bfi}~\Bigr\}
  $$
  are finite.
\end{lemma}

\subsection{Pseudo-basis and distance function.}
The recent paper \cite{RRW12} introduced several techniques such as
pseudo-basis, distance function and matchable relation. These techniques
allowed us to prove several theorems that could not be obtained using the
classical techniques.

The  paper \cite{RRW12} considered the notion of \emph{rank} for a
contraction vector $\brho=(\rho_1,\ldots,\rho_m)$. Let $\langle
\rho_1,\dots, \rho_m \rangle$ denote the subgroup of
$(\bR^+,\times)$ generated by $\rho_1$, $\dots$, $\rho_m$, it
is a free abelian group. It follows that $\langle \rho_1,\dots,
\rho_m \rangle$ has a nonempty basis and we can define the rank of
$\langle \rho_1,\dots, \rho_m \rangle$, denoted by
$\rank\langle \brho \rangle$, to be the cardinality of the basis.
Clearly $1\leq \rank\langle \brho \rangle \leq m$. If
$\rank\langle \brho \rangle=m$, we say $\brho$ has full rank. For more about the
rank of a free abelian group see e.g. \cite{Hun80}.

According to Theorem \ref{theo:FaMa} (2), if $\SD(\brho)\sim
\SD(\btau)$, then
$\rank\langle\brho\rangle=\rank\langle\btau\rangle=\rank\langle\brho,\btau\rangle$,
where
$\langle\brho,\btau\rangle:=\langle\rho_1,\ldots,\rho_m,\tau_1,\ldots,\tau_n\rangle$.

We call $w_1, \dots, w_L$ a \emph{pseudo-basis} of
$V=\langle \brho,\btau\rangle$ if $L=\rank\, V$ and $V\subseteq \langle
w_1,\ldots,w_L\rangle$.
%, and every entry of $\brho$ and $\btau$ has the form $w_1^{p_1}\dots w_L^{p_L}$ with
%$p_1,\dots, p_L\in {\mathbb Z}$.  The existence of pseudo-basis is guaranteed by Theorem \ref{theo:FaMa} (2).
It is clear that a basis of $V$ is natural to be a pseudo-basis. For any $x_1,x_2\in V$, we define
their distance with respect to the pesudo-basis $w_1,\ldots,w_L$ by
\begin{equation}\label{eq:def-h}
     h(x_1, x_2) := \sqrt{\sum_{j=1}^L(s_j-t_j)^2},
\end{equation}
where $s_j,t_j\in\bZ$ are the unique integers such that $x_1=
\prod_{j=1}^L w_j^{s_j}$, $x_2= \prod_{j=1}^L w_j^{t_j}$.

\begin{rem}\label{rem:distance}
\textnormal{It is easy
to show that if $h_1$ and $h_2$ are distances on $V$ with respect to two different pseudo-bases, then they are comparable, i.e., there exists a constant
$C\geq 1$ such that
$$
  C^{-1} h_1(x_1,x_2) \leq h_2(x_1,x_2) \leq C h_1(x_1,x_2), \quad
  \forall x_1,x_2\in V.
$$
}
\end{rem}
% Hence, we fix the pseudo-basis and the function $h$ from now on.

Let $\brho_{\max} = \max\{\rho_1, \dots, \rho_m\}$ and
$\brho_{\min} = \min\{\rho_1, \dots, \rho_m\}$. For any $t\in(0,1)$
let
$$
\SW(E, t):=\{\bfi\in\Sigma_n^*:  \brho_\bfi \leq t<\brho_{\bfi^*}\},
$$
where $\bfi^*$ is the word obtained by deleting the last letter of
$\bfi$, i.e., $\bfi^*=i_1\cdots i_{k-1}$ if $\bfi=i_1\cdots i_k$. We
define $\brho_{\bfi^*}=1$ if the length of $\bfi$ equals $1$.
Similarly, we may define $\SW(F, t)$ with respect to its contraction
vector $\btau$.

Pick some $\bfi\in\Sigma_m^*$. There is a (unique) maximum
decomposition of $f(E_\bfi)$ with respect to $F$ and $n_0$:
$$
   f(E_\bfi)=\bigcup_{r=1}^{p_\bfi} F_{\bfk\bfj_r},
$$
where  $|\bfj_r|= n_0$. We define a relation $\SR(\bfi,t,f)\subset
{\SW}(E, t) \times {\SW}(F, t)$ by
\begin{equation}\label{3.6}
    \SR(\bfi,t,f):=\left \{(\bfi',\bfj')\in {\SW}(E, t) \times
    {\SW}(F, t):~f(E_{\bfi\bfi'})\cap \bigcup_{r=1}^{p_\bfi} F_{\bfk\bfj_r\bfj'}
    \neq \emptyset. \right \}.
\end{equation}
%This relation plays an important role in the proof.

It is shown in \cite{RRW12} that

\begin{theorem}[\cite{RRW12}]    \label{theo:new-criterion}
     Assume that $f:~E\lra F$ is bi-Lipschitz and let
     $E_{\bfi_0}$ be a stable cylinder for some $\bfi_0\in\Sigma_m^*$.
     Let $h$ be a distance on $V=\langle\brho,\btau\rangle$ defined by (\ref{eq:def-h}).
     Then there exists a constant $M_0>0$ such that for any $t\in(0,1)$ we have
     \begin{itemize}
     \item[\rm (1)] For any  $\bfi\in \SW(E, t)$,
     \begin{equation} \label{MM}
     1\leq  \card \{\bfj:~(\bfi,\bfj)\in\SR(\bfi_0,t,f)\}\leq M_0.
     \end{equation}
     Similarly,  for any
           $\bfj\in \SW(F, t)$, $1\leq  \card \{\bfi:~(\bfi,\bfj)\in\SR(\bfi_0,t,f)\}\leq M_0.$
%      \vspace{2mm}
     \item[\rm (2)]~~If $(\bfi,\bfj)\in\SR(\bfi_0,t,f)$ then $h(\brho_\bfi,\btau_\bfj)\leq M_0$.
     \end{itemize}
\end{theorem}

\subsection{Matchable condition.} One of the most important techniques introduced
in \cite{RRW12} is the {\em matchable relation}. It is also one of the more
technical ones. Let $E$ and $F$ be two dust-like self-similar sets
with contraction vectors $\brho$ and $\btau$ respectively. Let $h$
be a distance on $V=\langle\brho,\btau\rangle$ defined by
(\ref{eq:def-h}).

Let $M_0$ be a constant. For $t\in(0,1)$, a relation $\SR\subset
\SW(E,t)\times \SW(F,t)$ is said to be \emph{$(M_0,h)$-matchable},
or simply {\em $M_0$-matchable} when there is no confusion,  if
\begin{itemize}
\item[(i)] $1\leq  \card \{\bfj:~(\bfi,\bfj)\in {\SR}\}\leq M_0 $ for any
$\bfi\in \SW(E,t)$, and $1\leq  \card \{\bfi:~(\bfi,\bfj)\in
\SR\}\leq M_0 $ for any $\bfj\in  \SW(F,t)$.
\item[(ii)]  If $(\bfi,\bfj)\in {\SR}$, then $h(\brho_i,\btau_j)\leq
M_0$.
\end{itemize}
We also say that $\SW(E,t)$ and $\SW(F,t)$ are
\emph{$(M_0,h)$-matchable}, or {\em $M_0$-matchable} when there exists a
$(M_0,h)$-matchable relation $\SR\subset \SW(E,t)\times \SW(F,t)$.

\begin{definition}\textnormal{
We shall call two self-similar sets  $E$ and $F$ are {\em
matchable}, if there exists a constant $M_0$ such that for any
$t\in(0,1)$, $\SW(E,t)$ and $\SW(F,t)$  are $M_0$-matchable.}
\end{definition}

We remark that the matchable property does not depend on the choice
of pseudo-basis of $\langle\brho,\btau\rangle$.
Obviously Theorem \ref{theo:new-criterion} implies the following result:

\begin{theorem}[\cite{RRW12}] \label{Match}  Let $E$ and $F$ be two dust-like self-similar sets.
If $E\sim F$, then $E$ and $F$ are matchable.
\end{theorem}

\section{Recent Results on dust-like self-similar sets}

The techniques developed in Falconer and Marsh \cite{FaMa92} had led to
some fundamental theorems on the Lipschitz equivalence of dust-like Cantor sets, such as
Theorem \ref{theo:FaMa}. However, to further advance the field these techniques are
clearly not sufficient. As a result there has not been much significant progress until recently,
when several new results on the Lipschitz equivalence of
dust-like Cantors sets were established in \cite{RRW12,Xi10,XiRu08}. In particular,
the equivalence of several classes have been completely characterized in \cite{RRW12}.
These results, which we shall state here, are based on
the new techniques outlined in the previous section. As an important observation,
a common theme among these
results is the link between Lipschitz equivalence and the algebraic properties of
the contractions.

%\subsection{Self-similar sets with full algebraic rank}

One of the new results on the equivalence of two dust-like Cantor sets concerns
the special case where one of the
contraction vectors has full rank. Lipschitz equivalence in
this setting forces strong rigidity on the contraction vectors. The following result
is derived by using the distance function and Theorem \ref{Match}.

\begin{theorem}[\cite{RRW12}] \label{theo:rankFull}
Let $\brho=(\rho_1,\ldots,\rho_m)$ and $\btau=(\tau_1, \ldots,
\tau_m)$ be two contraction vectors with  $\rank\langle
\brho\rangle=m$. Then   $\SD(\brho)$ and $\SD(\btau)$ are Lipschitz
equivalent if and only if $\btau$ is a permutation of $\brho$.
\end{theorem}
If the length of $\btau$ is not equal to $m$ then the characterization of $\btau$ is open.
We make the following conjecture:

\begin{conj}
Let $\brho=(\rho_1,\ldots,\rho_m)$ such that  $\rank\langle
\brho\rangle=m$. Assume that $\btau=(\tau_1, \ldots,
\tau_n)$. Then $\SD(\brho)$ and $\SD(\btau)$ are Lipschitz
equivalent if and only if $\btau$ is derived from $\brho$.
\end{conj}

%\subsection{Two-branch dust-like Cantor sets}
%\label{section:two-branch} \setcounter{equation}{0}

Another interesting and natural class to consider is when the contraction vectors
have two ratios. Namely we may ask under what conditions are
$\SD(\rho_1, \rho_2)\sim \SD(\tau_1,\tau_2)$. This question is completely
answered in \cite{RRW12}.

\begin{theorem} \label{theo:twobranch} %\cite{RRW12}
Let $(\rho_1, \rho_2)$ and $(\tau_1, \tau_2)$ be two contraction
vectors with $\rho_1\leq \rho_2$, $\tau_1\leq \tau_2$. Assume that
$\rho_1\leq \tau_1$. Then $\SD(\brho)\sim \SD(\btau)$ if and only if
one of the two conditions holds:
    \begin{itemize}
    \item[\rm (1)]~~ $\rho_1=\tau_1$ and $\rho_2=\tau_2$.
    \item[\rm (2)]~~There exists a real number $0<\lambda<1$,  such that
    $$
       (\rho_1,\rho_2) = (\lambda^5,\lambda) \quad \mbox{and}\quad (\tau_1,\tau_2)=  (\lambda^3,\lambda^2).
    $$
    \end{itemize}
\end{theorem}

We provide a quick sketch of the proof here. First, assume that
$\rank\langle\rho_1,\rho_2\rangle=2$ or $\rank\langle\tau_1,\tau_2\rangle=2$. Then
we must have $\rho_1=\tau_1$ and $\rho_2=\tau_2$ by Theorem \ref{theo:rankFull}.
So we now only need to consider the case where
$\rank\langle\rho_1,\rho_2\rangle=\rank\langle\tau_1,\tau_2\rangle=1$. By
Theorem \ref{theo:FaMa} we know there exists a $t$ such that
$\rho_j = t^{m_j}$ and $\tau_j = t^{n_j}$ where $m_j, n_j\in\bZ^+$.
Set $x=t^s$ where $s$ is the dimension of $\SD(\rho_1,\rho_2)$. Then
$$
    x^{m_1}+x^{m_2}-1=0, \mhsp x^{n_1}+x^{n_2}-1=0.
$$
For the above two polynomials to have a common root they must have a
common factor. The irreducibility of trinomials, however, has been classified
by Ljunggren \cite{Lju60} (Theorem~3 in the paper). Applying the results in
\cite{Lju60} one can show that
$$
       (\rho_1,\rho_2) = (\lambda^5,\lambda) \mhsp \mbox{and}\quad (\tau_1,\tau_2)=  (\lambda^3,\lambda^2)
$$
for some $0<\lambda<1$, which takes on the form $\lambda= t^k$ for some $k\in\bZ^+$.

As an application of Theorem \ref{theo:twobranch}, we can see that
the conditions in Theorem~\ref{theo:FaMa} are necessary but not
sufficient via the following example.

\begin{example}
Let $x,y$, $0<x,y<1$, be the solution of the equations
$$
x^6+y=1 \mbox{ and } x^3+y^4=1.
$$
One can easily check that the solution indeed exists. Let $s$ be a
real number such that $0<s<1$. Suppose that the contraction vectors
of $E$ and $F$ are $(x^{6/s}, y^{1/s})$ and $(x^{3/s}, y^{4/s})$,
respectively. Then $E$ and $F$ have the same Hausdorff dimension and
satisfy the conditions in Theorem~\ref{theo:FaMa}. However, $E$ and
$F$ are not Lipschitz equivalent by Theorem~\ref{theo:twobranch}.
\end{example}

%\subsection{Self-similar sets with uniform contraction ratio}
%\setcounter{equation}{0}

Another case where the Lipschitz equivalence of dust-like
self-similar sets can be characterized completely is when one of
them has uniform contraction ratios.

\begin{theorem}[\cite{RRW12}] \label{theo:uniform-contra}
Let $\brho=(\rho_1,\cdots,\rho_m)=(\rho,\dots,\rho)$ and
$\btau=(\tau_1, \dots, \tau_n)$. Then $\SD(\brho)$ and $\SD(\btau)$
are Lipschitz equivalent if and only if the following conditions
hold:
  \begin{itemize}
    \item[\rm(1)] $\dim_H \SD(\btau)= \dim_H \SD(\brho)=\log m/\log \rho^{-1}$.
    \item[\rm(2)] There exists a $q\in\bZ^+$ such that $m^{1/q}\in\bZ$ and
      $$
         \frac{\log \tau_j}{\log \rho} \in \frac{1}{q}\bZ \quad
      \mbox{ for all } \;\; j=1,2,\ldots,n.
      $$
  \end{itemize}
\end{theorem}
Note that by Theorem \ref{theo:FaMa} all $\tau_j$ must be rational powers of $\rho$.
The above theorem shows that one needs more to achieve Lipschitz equivalence.

In other direction, using a measure-preserving property, Xi and Ruan \cite{XiRu08} and Xi \cite{Xi10} showed that the graph-directed structure can be used to characterize the Lipschitz equivalence of two dust-like self-similar sets. We remark that the idea of studying graph-directed structures of self-similar sets appeared in \cite{RW}, where they deal with self-similar sets with overlaps.
%We will explain concepts related to graph-directed structure in detail in next section.

We recall the definition of
graph-directed sets (see \cite{MW}). Let
$G=(V,\Gamma)$ be a directed graph and $d$ a positive integer. Suppose for each edge $e\in \Gamma$,
there is a corresponding similarity $\phi_e:\, \bR^d\to \bR^d$ with
ratio $\gr_e\in (0,1)$. Assume that for each vertex $i\in V$, there exists an
edge starting from $i$. Then there exists a unique family
$\{E_i\}_{i\in V}$ of compact subsets of $\bR^d$ such that for any
$i\in V$,
\begin{equation}\label{eq:GDS-def-Ei}
  E_i=\bigcup_{j\in V} \bigcup_{e\in \SE_{ij}} \phi_e(E_j),
\end{equation}
where $\SE_{ij}$ is the set of edges starting from $i$ and ending at
$j$. In particular, if the union in (\ref{eq:GDS-def-Ei}) is disjoint
for any $i$, we call $\{E_i\}_{i\in V}$ \emph{dust-like graph-directed
sets} on $(V,\Gamma)$.

Now, let $\{F_i\}_{i\in V}$ be dust-like
  graph-directed sets on $(V,\Gamma)$ satisfying
\begin{equation}\label{eq:GDS-def-Fi}
  F_i=\bigcup_{j\in V} \bigcup_{e\in \SE_{ij}}\psi_e(F_j), \quad i\in V.
\end{equation}
  If similarities $\phi_e$ and $\psi_e$ have the same ratio for each $e\in\Gamma,$ we say that $\{E_i\}_{i\in V}$ and $\{F_i\}_{i\in V}$ have the \emph{same graph-directed structure}.

Recall that $E$ and $F$ are the attractors of the
IFSs $\Phi=\{\phi_1, \dots, \phi_m\}$ and $\Psi=\{\psi_1, \dots, \psi_n\}$, respectively. Given a finite subset $\Lambda$ of $\gS_n^*$ and a positive real number $r$, we call $r\cdot \bigcup_{\i\in \Lambda}\psi_{\i}(F)$ a \emph{finite copy} of $F$. It was proved in \cite{CP,XiRu08} that a finite copy of $F$ is always Lipschitz equivalent to $F$.

\begin{theorem}[\cite{Xi10,XiRu08}]
  Let $E$ and $F$ be two dust-like self-similar subsets of $\bR^d$. Then $E\sim F$ if and only if there exist graph-directed sets $\{E_i\}_{i=1}^\ell$ and $\{F_i\}_{i=1}^\ell$ such that
  \begin{enumerate}
    \item $\{E_i\}_{i=1}^\ell$ and $\{F_i\}_{i=1}^\ell$ have the same graph-directed structures,
    \item $E_i=E$ for $i=1,\ldots,\ell$,
    \item $F_i$ is a finite copy of $F$ for $i=1,\ldots,\ell$.
  \end{enumerate}
\end{theorem}
Notice that  the conditions in the above theorem
are often difficult to check. We pose the following problem.
\begin{ques}\label{ques-algo}
  Given two contraction ratios $\brho$ and $\btau$, devise an algorithm to determine
  in finite steps  the Lipschitz equivalence of $\SD(\brho)$ and $\SD(\btau)$.
\end{ques}

%\subsection{Remarks} Other conditions on Lipschitz equivalence of self-similar sets have
%been established, e.g. in Xi and Ruan \cite{XiRu08} and in Xi
%\cite{Xi09}. In both studies, sufficient and necessary conditions
%for Lipschitz equivalence have been established in terms of
%graph-directed sets. However, these conditions are difficult to
%check.
%Generally, given two contraction vectors
%$\brho=(\rho_1,\rho_2, \dots, \rho_m)$ and $\btau=(\tau_1,\tau_2,
%\dots,\tau_n)$, it is not practical to apply these conditions to
%decide whether $\SD(\brho)$ and $\SD(\btau)$ are Lipschitz equivalent,
%even for the two-branch case $m=n=2$.

\medskip

\section{Touching IFS and Lipschitz equivalence: One dimensional case}

So far we have focused almost exclusively on the algebraic properties of contraction ratios.
Yet we should not overlook the importance of geometry in the study.
One interesting question in Lipschitz equivalence concerns the geometric
structures of the generating IFSs of self-similar sets.
One such problem is the Lipschitz equivalence of
two self-similar sets have the same contraction ratios but one is dust-like while
another has some touching structures.
The best known example is Problem \ref{prob1} in Section 1, known as
the $\{1,3,5\}-\{1,4,5\}$ problem
proposed by David and Semmes (\cite{DS}, Problem~11.16).
%
%\begin{problem}\label{prob1}
%  Let $S_i(x):=x/5+(i-1)/5$ where $i\in \{1,\cdots,5\}$. Let $M$ and
%  $M^\prime$ be the attractor of the IFS $\{S_1,S_3,S_5\}$ and
%  the IFS $\{S_1,S_4,S_5\}$, respectively. Are $M$ and $M^\prime$ Lipschitz
%  equivalent?
%\end{problem}
As we mentioned in Section 1, this problem was settled in \cite{RRX06}, which
proved that the two sets are indeed Lipschitz equivalent. In this section we give a
more detailed description of the techniques used in \cite{RRX06} to solve the $\{1,3,5\}-\{1,4,5\}$ problem. These techniques have also led to further
recent development \cite{RWX,XiRu07} on the Lipschitz equivalence of touching IFSs in more general settings. We shall provide more details on these developments as well.

\subsection{The $\{1,3,5\}-\{1,4,5\}$ problem and the graph-directed method}

An important technique is the graph-directed method, and here we show how it works
by proving the equivalence of the sets $M$ and $M^\prime$. Recall from Section 1, Problem
\ref{prob1} that $M$ is the dust-like $\{1,3,5\}$-set while $M'$ is the
$\{1,4,5\}$-set, which has touching structure, see Figure~\ref{figure: DS-set}.

\begin{theorem}[\cite{RRX06}]\label{thm:gds-strong}
  Suppose that dust-like graph-directed sets $\{E_i\}_{i\in V}$ and $\{F_i\}_{i\in V}$ have the same graph-directed structure. Then $E_i\sim
  F_i$ for each $i\in V$.
\end{theorem}
\begin{proof}
We shall use the notations in \eqref{eq:GDS-def-Ei} and \eqref{eq:GDS-def-Fi}.
Since $\{E_i\}_{i\in V}$ are dust-like, for any $x\in E_{i}$,
there is a unique infinite path $%
e_{1}\cdots e_{k}\cdots $ starting at $i$ such that
$$
    \{x\}=\bigcap _{k=1}^\infty \phi_{e_{1}\cdots e_{k}}(E_{i_{k}})
$$
where the edge $e_{k}$ ends at $i_{k}$ for every $k$. We say that $e_1e_2\cdots$ is the coding of $x$.
Hence the mapping $f:\, E_i\to F_i$ defined by
$$ \{f(x)\}=\bigcap_{k=1}^\infty \psi_{e_{1}\cdots e_{k}}(F_{i_{k}}).
$$
is a bijection. It remains to show that $f$ is bi-Lipschitz.

Suppose $x,x^{\prime }\in E_{i}$. Let $e_1e_2e_3\cdots$ and
$e^\prime_1 e^\prime_2 e^\prime_3 \cdots$ be the coding of $x$ and $x^\prime$, respectively.
Let $m$ be the largest integer such that
$e_1e_2\cdots e_m = e^\prime_1 e^\prime_2\cdots e^\prime_m$.
Since both $x$ and $x^\prime$ are in the set $\phi_{e_1\cdots e_m}(E_{i_m})$, we have
$$ |x-x^{\prime }|\leq \diam \phi_{e_{1}\cdots
e_{m}}(E_{i_{m}})=\left(\prod_{i=1}^{m}\rho
_{e_{i}}\right) \diam(E_{i_{m}}).
$$
On the other hand, by the maximality of $m$, we have
\begin{eqnarray*}
|x-x^{\prime }| &\geq& d(\phi_{e_1\cdots e_m e_{m+1}}(E_{i_{m+1}}),
\phi_{e_1\cdots e_m e^\prime_{m+1}}(E_{i^\prime_{m+1}})) \\
&\geq& \left( \prod_{i=1}^m \gr_{e_i} \right) \min_{(e,e^\prime)} d(\phi_e(E_j), \phi_{e^\prime} (E_{j^\prime})),
\end{eqnarray*}%
where the minimum is taking over all the pairs $(e,e^\prime)$ of distinct edges stemming from a common vertex.
For such a pair, let $j$ and $j^\prime$ be the end vertices of $e$ and $e^\prime$.
Since $e$ and $e^\prime$ start from a common vertex $i$, $\phi_e(E_j)$ and $\phi_{e^\prime} (E_{j^\prime})$
are disjoint closed subsets of $E_i$. Hence the minimum is a positive number.

Therefore, there exists a constant $c_{1}>0$ depending only on
$\{E_{i}\}$ and $\{\phi_{e}\}$ such that $$
c_{1}^{-1}(\prod_{i=1}^{m}\rho _{e_{i}})\leq |x-x^{\prime
}|\leq c_{1}(\prod_{i=1}^{m}\rho _{e_{i}}).
$$
Similarly, there exists a constant $c_{2}>0$ depending only on
$\{F_{i}\}$ and $\{\psi_{e}\}$ such that
$$c_{2}^{-1}(\prod_{i=1}^{m}\rho _{e_{i}})\leq
|f(x)-f(x^{\prime })|\leq c_{2}(\prod_{i=1}^{m}\rho
_{e_{i}}).$$
It follows that
$c_{1}^{-1}c_{2}^{-1}|x-x^{\prime }|\leq
|f(x)-f(x^{\prime })|\leq c_{1}c_{2}|x-x^{\prime }|.
$
\end{proof}

\begin{rem}
Theorem~\ref{thm:gds-strong} and its proof are natural extensions of Proposition~\ref{prop-1.1}.
\end{rem}

With the above lemma we can show that the $\{1,3,5\}$-set and
the $\{1,4,5\}$-set are Lipschitz equivalent.

\begin{prop}[\cite{RRX06}]
  The $\{1,3,5\}$-set $M$ and the $\{1,4,5\}$-set $M'$ are Lipschitz equivalent.
\end{prop}
\begin{proof}
  Define $M_1=M, M_2=M \cup (M+2), M_3=M \cup (M+2) \cup (M+4)$, and
    $M^\prime_1=M^\prime, M^\prime_2=M^\prime \cup (M^\prime +1), M^\prime_3=M^\prime \cup (M^\prime +1) \cup (M^\prime +2)$.
   Clearly,
  \begin{align*}
   & M_1 = M_1/5 \cup (M_2/5 + 2/5),  \qquad  M_2 = M_1/5 \cup (M_3/5 + 2) \cup (M_2/5+ 2/5), \\
   & M_3 = M_1/5 \cup (M_3/5 +2) \cup (M_3/5+4) \cup (M_2/5+ 2/5),  \\
   & M_1^\prime=M_1^\prime/5 \cup (M^\prime_2/5+3/5),  \qquad   M_2^\prime=M_1^\prime/5 \cup (M^\prime_3/5+3/5) \cup (M^\prime_2/5 +8/5), \\
   & M^\prime_3 = M^\prime_1/5 \cup (M^\prime_3/5 +3/5) \cup (M^\prime_3/5+8/5) \cup(M^\prime_2/5+ 13/5).
  \end{align*}
  Since all the similitudes have ratio $1/5$, Theorem~\ref{thm:gds-strong} shows that $M_1\sim M^\prime_1$, i.e., $M\sim M^\prime$.
\end{proof}

The technique can be applied to prove a more general theorem. Assume that $\brho=(\rho_1,\ldots,\rho_n)$ is a
contraction vector (in $\bR$) with $n\geq 3$.
Let $\Psi=\{\psi_i(x)=\rho_i x+t_i\}_{i=1}^n$ be an IFS on $\bR$ satisfing
the following three properties:
\begin{itemize}
\item[\rm (1)]~The subintervals
$\psi_1([0,1]),\ldots,\psi_n([0,1])$ are spaced from left to right
without overlapping, i.e. their interiors do not intersect. This means the contraction ratio is ordered.
\item[\rm (2)]~There exists at least one $i\in\{1,2,\ldots,n-1\}$, such that the intervals
$\psi_i([0,1])$ and $\psi_{i+1}([0,1])$ are touching, i.e.,
$\psi_i(1)=\psi_{i+1}(0)$.
\item[\rm (3)]~The left endpoint of $\psi_1[0,1]$ is $0$ and the right endpoint
of $\psi_n[0,1]$ is $1$. This means the touching is regular.
\end{itemize}
Denote by $T$ the attractor of the IFS $\gY$. We call $T$ a \emph{(regular) touching self-similar set} with \emph{(ordered) contraction vector} $\brho$.
In this section, we will always assume that the touching self-similar set is regular and the contraction vector is ordered.

Denote by $\ST(\brho)$ the family of all touching self-similar sets with contraction vector $\brho$.
We have the following theorem:
\begin{theorem}[\cite{RRX06}]
  Assume that $\brho=(\rho_1,\ldots,\rho_n)=(\rho,\ldots,\rho)$. Then $T\sim D$ for every $T\in \ST(\brho)$ and $D\in \SD(\brho)$.
\end{theorem}

\subsection{Generalization of the $\{1,3,5\}-\{1,4,5\}$ problem}

A natural generalization of the $\{1,3,5\}-\{1,4,5\}$ problem is when the
contraction ratios are no longer uniform. That is, one may consider the Lipschitz equivalence of $D\in \SD(\brho)$ and $T\in \ST(\brho)$, where $\brho=(\rho_1,\rho_2,\rho_3)$ is a contraction vector in $\bR$. Unlike in the dust-like
setting, the order of the contractions does make a difference. A complete answer was
given in Xi and Ruan \cite{XiRu07}.  Somewhat surprisingly, it is shown that
$D$ and $T$ are Lipschitz equivalent if and only if $\log\rho_1/\log\rho_3$
is rational.

%\begin{problem}\label{prob2}
%  Let $\brho=(\gr_1,\gr_2,\gr_3)$ be a contraction vector
%  (in $\mathbb{R}$). Let $\gF_i(x)=\gr_i x+d_i$, $i\in \{1,2,3\}$, where
%  $d_1=0, d_3=1-\gr_3$ and $\gr_1<d_2<1-\gr_2-\rho_3$ (e.g.
%  $d_2=\gr_1+(1-\gr_1-\gr_2-\gr_3)/2$).
%  Let $\gY_1=\gF_1$, $\gY_3=\gF_3$ and $\gY_2(x)=\gr_2 x+t_2$ with
%  $t_2=1-\gr_2-\gr_3$.
%  Let $M_{\brho}$ and
%  $M^\prime_{\brho}$ be the attractor
%  of $\{\gF_1,\gF_2,\gF_3\}$  and
%  $\{\gY_1,\gY_2,\gY_3\}$, respectively. Are $M_{\brho}$
%   and $M^\prime_{\brho}$ Lipschitz
%  equivalent?
%\end{problem}

From this result, one naturally asks the following question.
\begin{ques}
  Let $\brho=(\gr_1,\gr_2,\gr_3)$ and $\btau=(\gr_1,\gr_3,\gr_2)$ be two contraction vectors. Let $T\in\ST(\brho)$ and $T^\prime\in \ST(\btau)$ have initial structure shown as in Figure~\ref{figure: ThreeB-set}. Under what conditions are $T$ and $T^\prime$ Lipschitz equivalent?
\end{ques}

\begin{figure}[htbp]
\begin{center}
  \scalebox{0.5}{\includegraphics{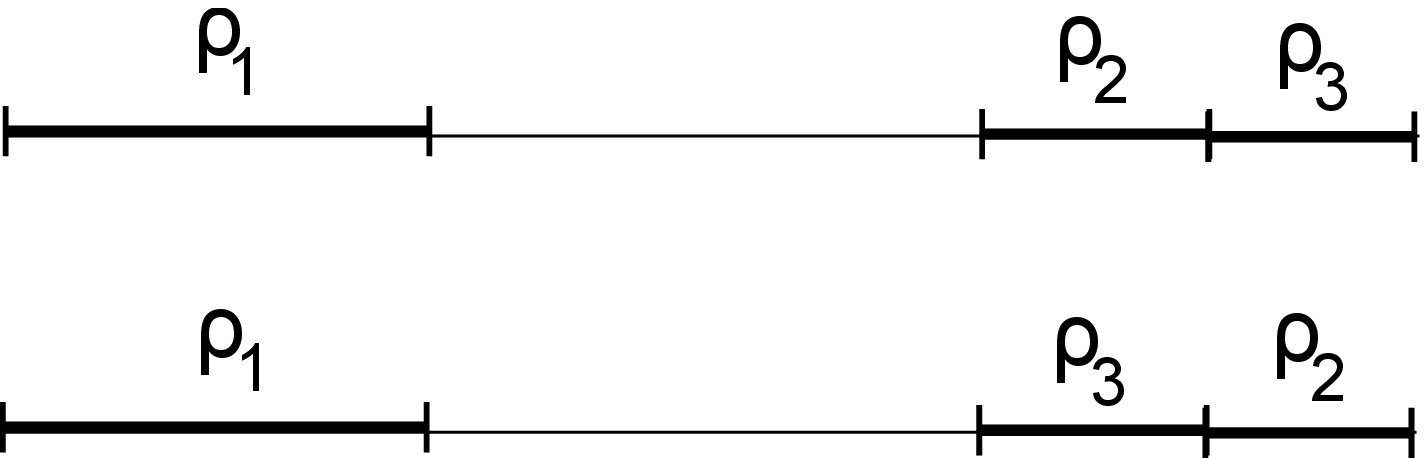}}
\end{center}
\caption{Basic intervals of the self-similar sets $T$ and $T^\prime$}
\label{figure: ThreeB-set}
\end{figure}

The result in \cite{XiRu07} is nevertheless a very special case. It is natural to
exploit such algebraic and geometric connections further in more
general settings. The proof in \cite{XiRu07} is quite complex, and allying it to
the more general setting appears to be very daunting. Recent work by
Ruan, Wang and Xi  \cite{RWX} has overcome some of the difficulties by
introducing a geometric notion called {\em substitutable}. It leads to
several results that provide insight into the problem.
%Here we provide a summary of these new results in \cite{RWX}. More details
%can be found in the paper.

Assume that $\brho=(\rho_1,\ldots,\rho_n)$ is a
contraction vector (in $\bR$) with $n\geq 3$. In the rest of this section, we assume that $D\in \SD(\brho)$, $T\in\ST(\brho)$ and $T$ is the attractor of an IFS $\gY=\{\psi_i(x)=\rho_i x+t_i\}_{i=1}^n$ on $\bR$.

A letter $i\in\{1,2,\ldots,n\}$
is a {\em (left) touching letter} if $\psi_i([0,1])$ and $\psi_{i+1}([0,1])$
are touching, i.e., $\psi_i(1) =\psi_{i+1}(0)$. We use
$\Sigma_T\subset\{1,\ldots,n\}$ to
denote the set of all (left) touching letters. For simplicity
we shall drop the word ``left'' for $\Sigma_T$. Let $\ga$ be the maximal integer such that $\bigcup_{i=1}^\ga \psi_i[0,1]$ is an interval. Similarly, let $\gb$ be the maximal integer such that
$\bigcup_{i=n-\gb+1}^n
\psi_i[0,1]$ is an interval.

Given a cylinder $T_\i$ and a nonnegative integer $k$, we can
define respectively the {\em level $(k+1)$ left touching patch} and the
{\em level $(k+1)$ right touching patch}
of $T_\i$ to be
\begin{equation}\label{eq:Lk-Rk-def}
  L_k(T_\i)=\bigcup_{j=1}^{\ga} T_{\i[1]^k j}, \quad
  R_k(T_\i)=\bigcup_{j=n-\gb+1}^n T_{\i [n]^k j},
\end{equation}
where $[\ell]^k$ is defined to be the word $\underbrace{\ell\cdots\ell}_k$
for any $\ell\in \{1,\ldots,n\}$, with $\i[1]^k j$ as the
concatenation of $\i$, $[1]^k$ and the letter $j$ (similarly for
$\i [n]^k j$). We remark that
$L_0(T_\i)=\bigcup_{j=1}^{\ga} T_{\i j}$ and $R_0(T_\i)=\bigcup_{j=n-\gb+1}^n T_{\i j}$.

A letter $i\in \Sigma_T$ is called \emph{left substitutable} if there exist $\j\in
\Sigma_n^*$ and $k,k'\in \bN$, such that $\diam L_{k}(T_{i+1})=\diam
L_{k'}(T_{i\j})$ and the last letter of $\j$
does not belong to $\{1\}\cup(\gS_T+1)$. Geometrically it simply means that a
certain left touching patch of the cylinder $T_{i+1}$ has the same diameter as
that of some left touching patch
of a cylinder $T_{i\j}$, and as a result we can substitute one of the left
touching patches by the
other without disturbing the other neighboring structures in $T$ because they
have the same diameter. Similarly, $i\in \Sigma_T$
is called \emph{right substitutable} if there exist $\j\in
\Sigma_n^*$ and $k,k'\in \bN$, such that $\diam R_{k}(T_{i})=\diam
R_{k'}(T_{(i+1)\j})$ and the last letter of $\j$
does not belong to $\{n\}\cup\gS_T$. We say that $i\in\Sigma_T$ is
\emph{substitutable} if it is left substitutable or right
substitutable.

\begin{rem}
  Both left and right substitutable properties can also be characterized algebraically.
  By definition, it is easy to check that
  $\diam L_{k}(T_{i+1})=\diam L_{k'}(T_{i\j})$ is
  equivalent to
  \begin{equation}\label{eq:left-disp}
    \gr_{i+1}\gr_1^k=\gr_i\gr_1^{k^\prime}\gr_\j,
  \end{equation}
  while $\diam R_{k}(T_{i})=\diam R_{k'}(T_{(i+1)\j})$ is equivalent
  to
  \begin{equation}\label{eq:right-disp}
    \gr_{i}\gr_n^k=\gr_{i+1}\gr_n^{k^\prime}\gr_\j.
  \end{equation}
\end{rem}

\begin{example}\label{exam2}
  Let $\brho=(\rho_1,\rho_2,\rho_3)$ with $\gS_T=\{2\}$. Then $\ga=1$ and $\gb=2$. Assume that $\log
  \gr_1/\log\gr_3\in\bQ$, i.e. there exist $u,v\in\bZ^+$ such that
  $\gr_1^u=\gr_3^v$. Pick $k=v+1$, $k^\prime=0$ and $\j=2[1]^u$. It
  is easy to check that (\ref{eq:right-disp}) holds for $i=2$ and the
  last letter of $\j$ is
  $1\not\in\{3\}\cup\gS_T$. Thus the touching letter $2$ is right substitutable.
\end{example}

Two main results of \cite{RWX} are listed as follows.

\begin{theorem}[\cite{RWX}]\label{thm:necessary-cond}
  Assume that $D\sim T$. Then $\log \gr_1/\log\gr_n\in\bQ$.
\end{theorem}

\begin{theorem}[\cite{RWX}]\label{thm:sufficient-cond}
  Assume that $\log \gr_1/\log \gr_n\in\bQ$. Then, $D\sim T$ if
  every touching letter for $T$ is substitutable.
\end{theorem}

Theorem~\ref{thm:sufficient-cond} allows us to establish a
more general corollary. The argument used to show the substitutability in
Example \ref{exam2} is easily extended to prove the following corollary:

\begin{corol}[\cite{RWX}]  \label{corol-1.3}
  $D\sim T$ if one of the following conditions holds:
  \begin{enumerate}
    \item[\rm (1)] ~$\log \gr_i/\log \gr_j\in\bQ$ for
  all $i,j\in\{1,n,\ga\}\cup (\gS_T+1)$.

  \item[\rm (2)] ~$\log \gr_i/\log \gr_j\in\bQ$ for
  all $i,j\in\{1,n,n-\gb+1\}\cup \gS_T$.
  \end{enumerate}
\end{corol}

The following result, which we state as a theorem because of the
simplicity of its statement, is a direct corollary of Corollary \ref{corol-1.3}.

\begin{theorem}[\cite{RWX}]\label{thm:suff-cond-log}
  Assume that $\log \gr_i/\log\gr_j \in \bQ$ for all
  $i,j\in \{1,\ldots,n\}$. Then $D\sim T$.
\end{theorem}

\section{Touching IFS and Lipschitz equivalence: Higher dimensional case}

Much of the work on Lipschitz equivalence with touching structure is set in $\mathbb{R}$. What about higher dimensions? While many of the results in $\mathbb{R}$ should generalize to
higher dimensions, some may not.

Let $Q=[0,1]\times [0,1]$ be the unit square. Given a positive integer $n\geq 3$
and a digit set $\SD\subset \{0,1,\ldots,n-1\}^2$,  there exists a unique nonempty compact $K\subset Q$ satisfying
\begin{equation*}
  K=\bigcup_{d\in \SD}\frac{1}{n} (K+d).
\end{equation*}
We denote the set $K$ by $K(n,\SD)$. Xi and Xiong \cite{XiXi10} obtained the following result.
\begin{theorem}[\cite{XiXi10}]
  Assume that $K(n, \SD_1)$ and $K(n, \SD_2)$ are totally disconnected. Then $K(n, \SD_1) \sim K(n, \SD_2)$ if and only if $\# \SD_1 = \# \SD_2 $.
\end{theorem}

Lau and Luo \cite{LauLuo12}, Roinestad \cite{Roi10}, and Wen, Zhu and Deng \cite{WZD12}
discussed the Lipschitz equivalence of $K(n, \SD_1)$ and $K(n, \SD_2)$  when at least
one of them has touching structure. However, unlike the one dimensional case,
$K$ may contain non-trivial connected components which makes the problem much harder.
\begin{ques}\label{ques-higher-dim}
  Establish necessary and sufficient conditions for the Lipschitz equivalence of
  $K(n, \SD_1)$ and $K(n, \SD_2)$. Clearly we must have $\# \SD_1 = \# \SD_2 $, but in
  general this is not sufficient. A simple case is $n=3$ and
  $$\SD_1=\{(0,0), (0,1), (0,2), (2,0), (2,2)\},$$
  $$\SD_2=\{(0,0), (0,1), (0,2), (2,1), (2,2)\}.$$ See Figure~\ref{figure: Block}. It is not known
  whether $K(3,\SD_1)\sim K(3,\SD_2)$.
\end{ques}

\begin{figure}[htbp]
\begin{center}
  \scalebox{0.5}{\includegraphics{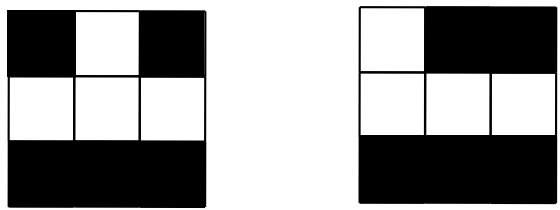}}
\end{center}
\caption{Initial structure of the self-similar sets $K(3,\SD_1)$ and $K(3,\SD_2)$}
\label{figure: Block}
\end{figure}

The sets discussed above are all self-similar. It is natural also to consider the Lipschitz equivalence of self-affine.

\begin{ques}
  What happen if the sets are self-affine but not self-similar? For example, when are McMullen carpets Lipschitz equivalent?
\end{ques}

Rao, Ruan and Yang \cite{RRY08} defined gap sequences for compact subsets in higher dimensional Euclidean space. \cite{RRY08} also proved that the gap sequence is a Lipschitz invariant. However, we do not know whether gap sequences can be used to prove that two self-similar sets (or self-affine sets) with the same Hausdorff dimension are not Lipschitz equivalent.

\noindent{\bf Acknowledgements.} The authors wish to thank D.-J.~Feng, L.-F.~Xi and Y.~Xiong for valuable discussions. Problem~\ref{ques-algo} comes from the discussion with  Feng, while we know the simple case in Problem~\ref{ques-higher-dim} from Xiong. We also wish to thank the referee for many helpful suggestions.

%    Text of article.

%    Bibliographies can be prepared with BibTeX using amsplain,
%    amsalpha, or (for "historical" overviews) natbib style.
\bibliographystyle{amsplain}

\begin{thebibliography}{10}

\bibitem{CP} D.~Cooper and T.~Pignataro, {\it On the shape of Cantor
sets}, J. Differential Geom., \textbf{28} (1988), 203--221.

\bibitem{DH} G.-T.~Deng and X.-G.~He, {\it Lipschitz equivalence of fractal sets in $\mathbb{R}$}, Sci. China. Math., to appear.

\bibitem{DS}
G.~David and S.~Semmes, {\it Fractured fractals and broken dreams:
self-similar geometry through metric and measure}, Oxford Univ.\
Press,~1997.

\bibitem{GLWX} Q.~Guo, H.~Li, Q.~Wang and L.~Xi, {\it Lipschitz equivalence of a class of self-similar sets with complete overlaps}, Ann. Acad. Sci. Fenn. Math., \textbf{37} (2012), 229--243.

\bibitem{FaMa92} K. J.~Falconer and D. T.~Marsh, {\it On the Lipschitz equivalence of
Cantor sets}, Mathematika, \textbf{39} (1992), 223--233.

\bibitem{Gromo07} M.~Gromov,
\newblock {\em Metric structures for {R}iemannian and non-{R}iemannian spaces}.
\newblock Modern Birkh\"auser Classics. Birkh\"auser Boston Inc., Boston, MA,
  english edition, 2007.
\newblock Based on the 1981 French original, With appendices by M. Katz, P.
  Pansu and S. Semmes, Translated from the French by Sean Michael Bates.


\bibitem{Hun80} T. W.~Hungerford, {\it Algebra}, Graduate Texts in Mathematics
{\bf 73}, Springer: New York, 1980.

\bibitem{Hut81} J. E. Hutchinson,  {\it Fractals and self-similarity},
Indiana Univ. Math. J.,  {\bf 30}  (1981),  713--747.

\bibitem{LauLuo12} K.-S.~Lau, and J.-J.~Luo, {\it Lipschitz equivalence of self-similar sets via hyperbolic boundaries}, preprint.



\bibitem{Lju60} W. Ljunggren, {\it On the irreducibility of certain trinomials and
quadrinomials}, Math. Scand., \textbf{8} (1960), 65--70.

%\bibitem{MS08}
%P. Mattila and P. Saaranen, {\it Ahlfors-David regular sets and
%bilipschitz maps}, Ann. Acad. Sci. Fenn. Math., 34 (2009), 487--502.

\bibitem{MW} R. D. Mauldin and S. C. Williams, {\it Hausdorff dimension in
graph directed constructions}, Trans. Amer. Math. Soc., 309 (1988),
811-829.

\bibitem{RRW12} H.~Rao, H.-J.~Ruan and Y. ~Wang, {\it Lipschitz equivalence of
Cantor sets and algebraic properties of constraction ratios}, Trans.
Amer. Math. Soc., \textbf{364} (2012), 1109--1126.


\bibitem{RRX06} H.~Rao, H.-J.~Ruan and L.-F.~Xi, {\it Lipschitz equivalence of self-similar
sets}, C. R. Acad. Sci. Paris. Ser. I, \textbf{342} (2006),
191--196.

\bibitem{RRY08} H.~Rao, H.-J.~Ruan and Y.-M.~Yang, {\it Gap sequence, Lipschitz
equivalence and box dimension of fractal sets}, Nonlinearity,
\textbf{6} (2008), 1339--1347.

\bibitem{RW} H.~Rao and Z.-Y.~Wen, {\it A class of self-similar fractals with overlap structure}, Adv. Appl. Math., \textbf{20} (1998), 50--72.

\bibitem{Roi10} Roinestad, K. A., {\it Geometry of fractal squares}, Ph.D. Thesis, Virginia Polytechnic Institute and State University, 2010.



\bibitem{RWX} H.-J.~Ruan, Y.~Wang and L.-F.~Xi,
{\it Lipschitz equivalence of self-similar sets with touching structures}, preprint (arXiv: 1207.6674v1 [math.MG]).


\bibitem{WZD12} Z.~Wen, Z.~Zhu and G.~Deng, {\it Lipschitz equivalence of a class of general Sierpinski carpets}, J. Math. Anal. Appl., \textbf{385} (2012), 16--23.

\bibitem{Xi10} L.-F.~Xi, {\it Lipschitz equivalence of dust-like self-similar
sets}, Math. Z., \textbf{266} (2010), 683--691.


\bibitem{XiRu07} L.-F.~Xi and H.-J.~Ruan, {\it Lipschitz equivalence of generalized
$\{1,3,5\}-\{1,4,5\}$ self-similar sets}, Sci. China Ser. A,
\textbf{50} (2007), 1537--1551.

\bibitem{XiRu08} L.-F.~Xi and H.-J.~Ruan, {\it Lipschitz equivalence of self-similar sets satisfying the strong separation
condition} (in Chinese), Acta Math. Sinica (Chin. Ser.), \textbf{51}
(2008), 493--500.

\bibitem{XiXi10} L.-F.~Xi, and Y.~Xiong, {\it Self-similar sets with initial cubic patterns}, C. R. Math. Acad. Sci. Paris, \textbf{348} (2010), 15--20.

\end{thebibliography}
%    Insert the bibliography data here.

\end{document}